\documentclass[12pt,reqno]{amsart}

\usepackage{amsmath, amssymb}
\usepackage[frame,cmtip,arrow,matrix,line,graph,curve]{xy}
\usepackage{graphpap,color,paralist,pstricks}
\usepackage[mathscr]{eucal}
\usepackage[pdftex]{graphicx}
\usepackage[pdftex,colorlinks,backref=page,citecolor=blue]{hyperref}
\usepackage{tikz}
\usepackage{kotex}
\usepackage{array}
\usepackage{pgfplots}
\usepackage{mathrsfs}
\usetikzlibrary{arrows} 
\usepackage{tikz-cd}
\usepackage{mathtools}
\usepackage{listings}

\newcolumntype{P}[1]{>{\centering\arraybackslash}p{#1}}
\newcolumntype{M}[1]{>{\centering\arraybackslash}m{#1}}

\newtheorem{theorem}{Theorem}[section]
\newtheorem{proposition}[theorem]{Proposition}

\newtheorem{lemma}[theorem]{Lemma}

\theoremstyle{definition}

\newtheorem{example}[theorem]{Example}

\newtheorem{remark}[theorem]{Remark}

\newcommand{\PP}{\mathbb{P}}

\newcommand{\CC}{\mathbb{C}}
\newcommand{\FF}{\mathbb{F}}

\newcommand{\ZZ}{\mathbb{Z}}

\newcommand{\cO}{\mathcal{O} }

\newcommand{\cE}{\mathcal{E} }
\newcommand{\cF}{\mathcal{F} }

\newcommand{\cH}{\mathcal{H} }
\newcommand{\cI}{\mathcal{I} }

\newcommand{\cM}{\mathcal{M} }

\newcommand{\cL}{\mathcal{L} }
\newcommand{\cQ}{\mathcal{Q} }

\newcommand{\cU}{\mathcal{U} }

\newcommand{\enm}[1]{\ensuremath{#1}}

\newcommand{\cal}[1]{\mathcal{#1}}

\renewcommand{\bar}[1]{\overline{#1}}

\newcommand{\Oo}{\enm{\cal{O}}}

\newcommand{\rH}{\mathrm{H} }

\newcommand\bG{\mathbf{G}}
\newcommand\bH{\mathbf{H}}

\newcommand\bR{\mathbf{R}}

\newcommand\coker{\mathrm{coker}}
\newcommand\Pic{\mathrm{Pic}}

\newcommand{\fsl}{\mathfrak{sl}}

\def\Sym{\mathrm{Sym} }
\def\Hom{\mathrm{Hom} }
\def\Ext{\mathrm{Ext} }

\def\Gr{\mathrm{Gr} }

\def\SL{\mathrm{SL}}

\def\lr{\rightarrow}

\newcommand{\ses}[3]{0\lr{#1}\lr{#2}\lr{#3}\lr 0}

\hypersetup{%
citecolor=blue,%
linkcolor=blue,%
}

\begin{document}

\title{Quartic curves in the quintic del Pezzo threefold}
\date{\today}

\author{Kiryong Chung}
\address{Department of Mathematics Education, Kyungpook National University, 80 Daehakro, Bukgu, Daegu 41566, Korea}
\email{krchung@knu.ac.kr}

\author{Jaehyun Kim}
\address{Research Institute for Mathematical Sciences, Kangwon National University, 1 Kangwondaehak -gil, Chuncheon-si, Gangwon-do, Korea}
\email{kjh6691@kangwon.ac.kr}

\author{Jeong-Seop Kim}
\address{Department of Mathematics Education, Sunchon National University, 255 Jungang-ro, Suncheon-si, Jeollanam-do 57922, Republic of Korea}
\email{jeongseop@scnu.ac.kr}


\begin{abstract}
In this paper, we prove that the Hilbert scheme $\mathbf{H}_4(X_5)$ of rational quartic curves on the quintic del Pezzo threefold $X_5$ is isomorphic to a Grassmannian bundle over the Hilbert scheme of lines on $X_5$.
In particular, $\mathbf{H}_4(X_5)$ is smooth and irreducible.
Our approach builds upon the geometry of rational quartic curves on $X_5$ studied by Fanelli--Gruson--Perrin in their work on the moduli space of stable maps to $X_5$.
%Let $X$ be the quintic del Pezzo threefold. By adjunction formula, the general intersection $X$ with a linear subspace $H$ of codimension two is an elliptic quintic curve $E_5$. If we choose the linear subspace $H$ containing a line lying on $X$, then $E_5$ is the union of that line and a rational quartic curve meeting at two points. In this paper, we prove that each rational quartic curve takes arise in this way even though the curve may not be reducible. In terms of moduli language, we prove that the Hilbert scheme of rational quartic curves in the del Pezzo threefold is isomorphic to a Grassmannian bundle over the Hilbert scheme of lines in $X$.
\end{abstract}
%JH: elliptic quintic은 E_5로, elliptic quartic은 E로 표기하였습니다.

%제목: Quartic curves in the quintic del Pezzo threefold

%초록: Let $X$ be the quintic del Pezzo threefold. By adjunction formula, the general intersection $X$ with linear subspace $H$ of codimension two is an elliptic quintic curve $E_5$. If we choose the linear subspace $H$ containing a line lying on $X$, then $E_5$ is the union of that line and rational quartic curve meeting at two points. In this talk, we prove that each rational quartic curve takes arise in this way even though the curve may not be reducible. This is a parallel study with that of arXiv:2412.17721 and is a joint work with Jaehyun Kim and Jeong-Seop Kim.

\maketitle
%\tableofcontents
\section{Introduction and results}
We work over the field $\CC$ of complex numbers.

\subsection{Related works and the result} Fano varieties are covered by families of rational curves, which play an essential role in understanding the geometry of the underlying varieties.
For example, lines and conics are used in Iskovskikh's classification of Fano threefolds \cite{IP99}, and the moduli space of lines is used by Clemens and Griffiths to disprove the rationality of cubic threefolds \cite{CG72}.
On the other hand, the geometry of the moduli spaces of rational curves is interesting in its own right, providing new examples of intriguing varieties (cf. \cite{BD85}), and their compactifications draw significant attention in the context of enumerative geometry (cf. \cite{FP97}).

Let $X$ be a projective variety with fixed embedding $X\subset \PP^r$ and $\bH_d(X)$ be the Hilbert scheme parametrizing curves $C$ in $X$ with Hilbert polynomial $\chi(\Oo_C(m))=dm+1$. For smooth Fano threefolds $X$ with minimal Betti numbers $b_2(X)=1$ and $b_3(X)=0$, some studies have been done for rational curves on $X$ with lower degree $d$.
The geometry of various compactifications of the moduli of rational curves on Fano varieties $X$ has been extensively studied by many researchers, including the first-named author --- after studies of various compactifications of the space of twisted cubics in the projective space $\PP^3$ \cite{PS85, VX02}, their geometries were compared in \cite{CK11}. The case of $X=\PP^3$ and $d=4$ is also studied in \cite{CCM16}.
For studies on the quadric threefold $Q^3$, we refer not only to \cite{CHY23} but also to \cite{Per02} in the context of rational curves on homogeneous varieties, and to \cite{HRS04} in the context of rational curves on hypersurfaces.
In addition, we may refer to \cite{Ili94, San14, CHL18, Chu22} for the quintic del Pezzo threefold $X_5$, and to \cite{MU83, Kuz97} for the prime Fano threefolds $V_{22}$ of degree $(-K_{V_{22}})^3=22$.
This paper is devoted to the study of the case where $X=X_5$ and $d=4$, a parallel continuation of the authors' previous work on the case where $X$ is the Mukai-Umemura variety $U_{22}$ and $d=4$ \cite{CKK24+}.
In this paper, contrary to \cite{CKK24+}, which only describes the local geometry, we succeed in describing the global geometry of a space of rational curves as follows.

\begin{theorem}
Let $\bH_4(X_5)$ be the Hilbert scheme that parametrizes curves $C$ on the quintic del Pezzo threefold $X_5$ with Hilbert polynomial $\chi(\Oo_C(m))=4m+1$.
Then $\bH_4(X_5)$ is isomorphic to a~$\Gr(3,5)$-bundle over $\PP^2$.
\end{theorem}

This irreducibility result for $X=X_5$ contrasts with the result for $X=\PP^3$, as $\bH_4(\PP^3)$ is known to have several irreducible components \cite[Theorem 4.9]{CN12}.
Meanwhile, the irreducibility is natural in light of the authors' previous result \cite{CKK24+} on the smoothness of $\bH_4(U_{22})$ for $X=U_{22}$, as the geometry of the underlying variety $X$ becomes more complicated, the space $\bH_d(X)$ of rational curves tends to simplify.

From Prokhorov’s classification of hyperplane sections of $X_5$ \cite[Section 2]{Pro92} together with Fanelli-Gruson-Perrin's description of a Kontsevich moduli space by \cite[Section 4.1]{FGP19}, we observe that all quartic rational curves arise as quartic curves contained in a hyperplane section of $X_5$, residual to a line in the same hyperplane section. Consequently, we can generate all residual quartic curves as ideal quotients of the ideals of the linear sections by those of lines.
Applying this fact, we construct an \emph{injective} morphism from a Grassmannian $\Gr(3,5)$-bundle over $\bH_1(X_5) = \PP^2$ to the Hilbert scheme $\bH_4(X_5)$ (Proposition~\ref{bij_morph}). This morphism turns out to be an isomorphism, as we show that $\bH_4(X_5)$ is a smooth, irreducible variety of dimension $8$ (Proposition~\ref{pro:sm} and Proposition~\ref{irred}). A key ingredient in the proof of irreducibility is to use a result of Lehmann--Tanimoto \cite[Theorem 1.5]{LT19}.
They prove that there is a unique irreducible component of the space of stable maps birational onto their image. By analyzing the component of multiple curves, we conclude that the Hilbert scheme $\bH_4(X_5)$ is irreducible.
%In verifying the smoothness, we use the $\CC^*$-action on $\bH_4(X_5)$ induced by the standard $\CC^*$-action on the $\SL_2$-invariant presentation of $X_5$ \cite{MU83}.
%We first identify $30$ fixed elements in the space $\bH_4(X_5)$, and for each type, we compute the dimensions of the tangent spaces.
%This allows us to conclude the smoothness by applying Białynicki-Birula’s theorem.
%Notably, thanks to the geometric interpretation of $X_5$, which leads us to use many geometric descriptions of the Grassmannian $\Gr(2,5)$, we do not rely on any Computer Algebra System such as Macaulay 2 at this step, which marks a key distinction from the authors' previous work on $U_{22}$ \cite{CKK24+}.

%It is worth mentioning that there is only one irreducible rational quartic curve fixed by the prescribed $\CC^*$-action on $X_5$.
%Besides this one among the $30$ types, there exist reducible and nonreduced quartic curves whose supports coincide with the union of two identical lines.

For further work, the authors expect a similar study to be conducted on higher dimensional smooth Fano varieties, such as other linear sections of the Grassmannian $\Gr(2,5)$.
Additionally, as an application of this paper, the authors expect it to serve as a foundation for computing the Donaldson--Thomas invariants (cf. \cite{CT21, CLW24}).

\subsection{Stream of the paper} This paper is organized as follows.  
In Section \ref{sec:predel}, we review the construction of $X_5$ by comparing its Plücker embedding and $\SL_2$-orbit coordinates.  
In Section \ref{sec:rat}, we summarize previous results on the Hilbert scheme $\bH_d(X_5)$ for $d\leq 3$, adding additional observations regarding the $\CC^*$-actions on $X_5$.  
In Section \ref{sec:raell}, after investigating the geometry of curves of genus $g\leq 1$ in $X_5$, we construct an injective morphism between $\bH_4(X_5)$ and a Grassmannian bundle over $\PP^2$.  
In Section \ref{sec:georatio}, we prove that the prescribed injective morphism is an isomorphism, which is our main result. 

\subsection*{Acknowledgements}
 The authors gratefully acknowledge the many helpful suggestions and comments provided by Jinwon Choi and Wanseok Lee during the preparation of this paper. The third-named author is especially grateful to Yuri Prokhorov for his valuable comments on \cite[Section 2]{Pro92}.

\section{$\SL_2$-invariant presentations of the del Pezzo variety}
\label{sec:predel}
In this section, we review the two well-known constructions of the quintic del Pezzo threefold and their relation.
\subsection{Pl\"{u}cker v.s. orbit coordinates of $X_5$}
Let $X_5$ be the quintic del Pezzo $3$-fold with $\SL_2$-invariant setting. We present $X_5$ in two ways (That is, the Pl\"{u}cker coordinates and a closure of $\SL_2$-orbits) and find its relations explicitly. For the presentation of the Pl\"{u}cker coordinates of $X_5$, let us denote the vector space $V_5=\Sym^4\CC^2=\langle e_4, e_2, e_0, e_{-2}, e_{-4}\rangle$ with weights $e_k=-k$ for $k\in \{\pm 4, \pm 2, 0\}$. There exists a unique $\SL_2$-invariant $7$-dimensional subspace in \[\wedge^2 V_5\cong \Sym^6\CC^2\oplus \Sym^2\CC^2.\]  Let us denote $p_{ij}$ by the Pl\"{u}cker coordinates of $\PP(\wedge^2V_5)$ where $p_{ij}(e_k\wedge e_l)=\delta_{ik}\delta_{jl}$ for $i <j$, $k<l$. The following were taken in \cite[Corollary 5.5.11 and Lemma 5.5.12]{CS15} (cf. \cite[Section 2.1]{San14}). The net of skew-symmetric forms in $\wedge^2 V_5^*$ is defined by
\[
3e_2^*\wedge e_0^*-e_4^*\wedge e_{-2}^*,\, 2e_2^*\wedge e_{-2}^*-e_4^*\wedge e_{-4}^*, 3e_0^*\wedge e_{-2}^*-e_{2}^*\wedge e_{-4}^*,
\]
where the operator $f\in \fsl_2(\CC)$ acts on this subspace in a standard way. Hence, the del Pezzo variety $X_5$ in $\PP(\wedge^2 V_5)=\PP^9$ is defined by the equations:
\begin{equation}\label{pluckeq}
I_{\Gr(2, V_5)/\PP^9}+\langle 3p_{2,0}-p_{4, -2}, 2p_{2,-2}-p_{4,-4}, 3p_{0,-2}-p_{2,-4} \rangle.
\end{equation}
\begin{lemma}\label{inv7dim}
The unique $\SL_2$-invariant $7$-dimensional subspace in $\wedge^2 V_5$ is generated by 
\[
e_4\wedge e_2,\,e_4\wedge e_0,\,e_2\wedge e_0+3e_4\wedge e_{-2},\, e_2\wedge e_{-2}+2e_4\wedge e_{-4},\, e_0\wedge e_{-2}+3e_2\wedge e_{-4},\,e_{0}\wedge e_{-4}, e_{-2}\wedge e_{-4}.
\]
\end{lemma}
\begin{proof}
Note that the operation $f$ acts on the dual space $V_5^{**}=V_5$ by
\[ f{\cdot}e_{-4}=e_{-2},\,f{\cdot}e_{-2}=2e_{0},\, f{\cdot}e_{0}=3e_{2},\, f{\cdot}e_{2}=4e_{4}.\]
By acting the operator $f$ successively to the obvious highest weight $e_{-2}\wedge e_{-4}$, we obtain the result.
\end{proof}
On the other hand, let us denote by $\bR_n=\CC[x, y]_{n}$ the vector space of degree $n$ homogeneous polynomials with respect to $x$, $y$. For an element $g(x, y)\in \bR_n$, the coefficient $a_i$ of \[g(x,y)=\displaystyle{\sum_{k=0}^{n}} a_{2k-n} \binom{n}{k}x^k y^{n-k}\] is taken as the homogeneous coordinate of $\PP^n$. Let $f(x,y)=xy(x^4-y^4)$ be a degree six homogeneous polynomial with distinct roots. It is well-known that the quintic del Pezzo variety $X_5$ is defined by the closure of the orbit of the closed point $f(x,y)$ by the action $\SL_2(\CC)$ (\cite{MU83}):
\[
X_5=\overline{\SL_2(\CC)\cdot f(x,y)}\subset \PP(\bR_6).
\]
In this case, the defining equations of $X_5$ are given by five quadric polynomials:
\begin{equation}\label{oribeq}
\begin{split}
\langle
 &a_{6}a_{-2}-4a_{4}a_{0 }+3a_{2}^2,\ 
 a_{6}a_{-4}-3a_{4}a_{-2}+2a_{2}a_{0},\ 
 a_{6}a_{-6}-9a_{2}a_{-2}+8a_{0}^2,\\
&a_{4}a_{-6}-3a_{2}a_{-4}+2a_{0}a_{-2},\ 
 a_{2}a_{-6}-4a_{0}a_{-4}+3a_{-2}^2
\rangle.
\end{split}
\end{equation}

\begin{proposition}\label{cochang}
The Pl\"{u}cker coordinates and $\SL_2$-orbit coordinates of the del Pezzo threefold $X_5$ defined in \eqref{pluckeq} and \eqref{oribeq}, respectively, are related by
\begin{gather*}
   p_{ 4, 2}= a_{ 6},
\; p_{ 4, 0}=2a_{ 4},
\; p_{ 2, 0}= a_{ 2},
\; p_{ 4,-2}=3a_{ 2},
\; p_{ 2,-2}=2a_{ 0},\\
   p_{ 4,-4}=4a_{ 0},
\; p_{ 0,-2}= a_{-2},
\; p_{ 2,-4}=3a_{-2},
\; p_{ 0,-4}=2a_{-4},
\; p_{-2,-4}= a_{-6}.
\end{gather*}
That is, there is a linear embedding of $\PP^6=\PP(\Sym^6\CC^2)$ into the space $\PP(\wedge^2V_5)$ such that the quintic del Pezzo variety defined by $\eqref{oribeq}$ corresponds to the variety defined by \eqref{pluckeq}.
\end{proposition}
\begin{proof}
Let $w_i$ be the standard dual vector of $a_i$. Let us associate the ordinates of $\Sym^6\CC^2$ to the vector space $\wedge^2V_5$ in the following way (cf. Lemma \ref{inv7dim}):
\[\begin{split}
 w_{ 6}\mapsto e_4\wedge e_2,\ 
&w_{ 4}\mapsto 2 e_4\wedge e_0,\ 
 w_{ 2}\mapsto  e_2\wedge e_0+3e_4\wedge e_{-2},
 w_{ 0}\mapsto 2(e_2\wedge e_{-2}+2e_4\wedge e_{-4}),\\ 
&w_{-2}\mapsto e_0\wedge e_{-2}+3e_2\wedge e_{-4},\ 
 w_{-4}\mapsto 2e_0\wedge e_{-4},\ 
 w_{-6}\mapsto e_{-2}\wedge e_{-4}.
\end{split}\]
Then one can easily check that the defining equation \eqref{pluckeq} has been  changed into $\eqref{oribeq}$.
\end{proof}
%R=QQ[p_(4,2),p_(4,0),p_(4,-2),p_(4,-4),p_(2,0),p_(2,-2),p_(2,-4),p_(0,-2), p_(0,-4),p_(-2,-4),a_0..a_6]
%I=ideal(a_0*a_4-4*a_1*a_3+3*a_2^2,a_0*a_5-3*a_1*a_4+2*a_2*a_3,a_0*a_6-9*a_2*a_4+8*a_3^2,a_1*a_6-3*a_2*a_5+2*a_3*a_4,a_2*a_6-4*a_3*a_5+3*a_4^2)---- SL2setting
%J=ideal(3*p_(2,0)-p_(4,-2), 2*p_(2,-2)-p_(4,-4), 3*p_(0,-2)-p_(2,-4),p_(4,2)*p_(0,-2)+p_(2,0)*p_(4,-2)-p_(4,0)*p_(2,-2),p_(4,2)*p_(0,-4)+p_(2,0)*p_(4,-4)-p_(4,0)*p_(2,-4),p_(4,2)*p_(-2,-4)+p_(2,-2)*p_(4,-4)-p_(4,-2)*p_(2,-4),p_(4,0)*p_(-2,-4)+p_(0,-2)*p_(4,-4)-p_(4,-2)*p_(0,-4),p_(2,0)*p_(-2,-4)+p_(0,-2)*p_(2,-4)-p_(2,-2)*p_(0,-4))---Pluckersetting
%K=ideal(a_0-p_(4,2), 2*a_1-p_(4,0), a_2-p_(2,0), 3*a_2-p_(4,-2),2*a_3-p_(2,-2), 2*2*a_3-p_(4,-4), a_4-p_(0,-2), 3*a_4-p_(2,-4), 2*a_5-p_(0,-4), a_6-p_(-2,-4))---Embeddingmap
%S=eliminate(J+K,{p_(4,2),p_(4,0),p_(4,-2),p_(4,-4),p_(2,0),p_(2,-2),p_(2,-4),p_(0,-2), p_(0,-4),p_(-2,-4)})
%I==S
\section{Rational curves of degree $\leq 3$ in $X_5$}\label{sec:rat}

In this section, we present the $\CC^*$-fixed rational curves of degree $d\leq 3$ in $X_5$ by using several Schubert varieties of $\Gr(2, V_5)$. From now on, let us fix the codimension three linear subspace $P$, which is defined by the linear forms in \eqref{pluckeq}. That is, $X_5=\Gr(2, V_5)\cap P$.
%\begin{lemma}
%Let $\PP(v)\in \PP(\wedge^2V_5)$ be a point for $v\in \wedge^2V_5$. Then $\PP(v)\in \PP(\wedge^2 V_4)$ for some $V_4<V_5$, $\dim V_i=i$.
%\end{lemma}
%\begin{proof}
%If $v$ is decomposable, then the result is obvious. If $v$ is non-decomposable, we may write the vector $v$ as
%\[
%v=w_0\wedge w_1 +w_2\wedge w_3 +w_4\wedge w_5
%\]
%such that $w_i\in V_5$, $0\leq i \leq 5$ and $\{w_0, w_1, w_2, w_3, w_4\}$ is a basis of $V_5$. In this basis, we can write the vector $w_5$ as a linear combination of $w_0, w_1, w_2, w_3$ since the contraction $w_4\wedge(-)$ cancels out the term $w_4$. Then by rewriting the last term $w_4\wedge w_5$ of $v$ as the combination of $w_0\wedge w_1$, $w_2\wedge w_3$, we may assume that
%\[
%w=s_0\wedge s_1+s_2\wedge s_3
%\]
%for the linearly independent vectors $s_0, s_1, s_2, s_3$ in $V_5$. We finish the proof by letting $V_4=\langle s_0, s_1, s_2, s_3\rangle$.
%\end{proof}

\subsection{A torus action on Grassmannian}
Since the following easy lemma will be used several times, we put the details of the proof here for the convenience of readers. For a vector space $V$, a vector $w\in \wedge^2V$ is called \emph{decomposable} if $w=v_1\wedge v_2$ for some $v_1$, $v_2\in V$ (equivalently, $\PP(w)\in \Gr(2, V)$). A vector $w$ is decomposable if and only if $w\wedge w=0\in \wedge^4V$.
\begin{lemma}\label{flocp}
Let $V$ be the finite-dimensional vector space with the diagonal $\CC^*$-action of distinct weights. Then the fixed loci of the Grassmannian variety $\Gr(2, V)$ are isolated whenever each weight occurs with multiplicity at most two among the Pl\"{u}cker coordinates in $\PP(\wedge^2V)$.
\end{lemma}
\begin{proof}
By the assumption, the fixed loci in $\PP(\wedge^2 V)$ are coordinate points when the weight in $\wedge^2 V$ is unique. If the weight of two vectors $e_i\wedge e_{j}$ and $e_k\wedge e_{l}$, $\{i, j\}\cap \{k, l\}=\emptyset$ is the same one, then any vector $ue_i\wedge e_{j}+ve_k\wedge e_{l}$, $(u, v)\in \CC^2$ is a $\CC^*$-fixed one. However, the line $\PP(ue_i\wedge e_{j}+ve_k\wedge e_{l})$ intersects with $\Gr(2, V)$ at two points $\PP(e_i\wedge e_{j})$, $\PP(e_k\wedge e_{l})$ because $w=ue_i\wedge e_{j}+ve_k\wedge e_{l}$ is a decomposable vector if and only if $u=0$ or $v=0$, and thus the fixed loci are isolated.
\end{proof}
%\begin{lemma}
%Let $\PP(v)\in \PP(\wedge^2V_5)$ be a point for $v\in \wedge^2V_5$. Then $\PP(v)\in \PP(\wedge^2 V_4)$ for some $V_4<V_5$, $\dim V_i=i$.
%\end{lemma}
%\begin{proof}
%If $v$ is decomposable, then the result is obvious. If $v$ is non-decomposable, we may write the vector $v$ as
%\[
%v=w_0\wedge w_1 +w_2\wedge w_3 +w_4\wedge w_5
%\]
%such that $w_i\in V_5$, $0\leq i \leq 5$ and $\{w_0, w_1, w_2, w_3, w_4\}$ is a basis of $V_5$. In this basis, we can write the vector $w_5$ as a linear combination of $w_0, w_1, w_2, w_3$ since the contraction $w_4\wedge(-)$ cancels out the term $w_4$. Then by rewriting the last term $w_4\wedge w_5$ of $v$ as the combination of $w_0\wedge w_1$, $w_2\wedge w_3$, we may assume that
%\[
%w=s_0\wedge s_1+s_2\wedge s_3
%\]
%for the linearly independent vectors $s_0, s_1, s_2, s_3$ in $V_5$. We finish the proof by letting $V_4=\langle s_0, s_1, s_2, s_3\rangle$.
%\end{proof}

\subsection{Fixed lines in $X_5$}\label{sub:fixl}
For an embedding of a smooth projective variety $X\subset \PP^r$, let $\bH_d(X)$ be the Hilbert scheme parametrizing curves $C$ in $X$ with Hilbert polynomial 
\begin{equation*}
 \chi(\cO_C(m))=dm+1.
\end{equation*}
It is known that the Hilbert scheme of lines in $\Gr(2, V_5)$ is isomorphic to $\bH_1(\Gr(2, V_5))\cong \Gr(1, 3, V_5)$ by \cite[Exercise 6.9]{Har95}. A line in $\Gr(2, V_5)$ is determined by a pencil of lines passing through a point (so-called \emph{vertex}). Also, under the natural embedding $\bH_1(X_5)\subset \bH_1(\Gr(2, V_5))$, we have a $\CC^*$-equivariant projection map
\[
p:\bH_1(X_5)\subset\bH_1(\Gr(2, V_5))\cong \Gr(1, 3, V_5)\lr \Gr(1, V_5)
\]
which associates a line to its \emph{vertex}. Clearly, the fixed locus of $\Gr(1, V_5)$ is just the coordinate points. For each point $\PP(e_i)$, $i\in \{\pm 4, \pm 2, 0\}$, the fixed point of the fiber $p^{-1}(\PP(e_i))$ determines the fixed loci of $\bH_1(X_5)$. However, the fiber of the map $p$ is isomorphic to $\Gr(2, V_5/\langle e_i\rangle)$. Note that the weights of the quotient space $V_5/\langle e_i\rangle$ are distinct from each other. Thus by Lemma \ref{flocp}, the $\CC^*$-fixed lines in $\Gr(2, V_5)$ are of the form:
\[
\{e_i\wedge(ue_s+ ve_t)\}, [u:v]\in \PP^1
\]
for $s, t\in \{\pm 4, \pm 2, 0\}\setminus \{i\}$. By plugging the above lines into the defining relation \eqref{pluckeq} of the del Pezzo variety, we can check that there exist exactly three fixed lines in $X_5$ as follows.
\begin{equation}\label{fixedlines}
   e_4\wedge(ue_2+ve_0), e_2\wedge(ue_0+ve_{-2}), e_0\wedge(ue_{-2}+ve_{-4}), [u:v]\in \PP^1.
\end{equation}
\begin{remark}
Under $\SL_2$-orbit coordinates of the del Pezzo variety $X_5$, lines in $X_5$ are given by (\cite[Lemma 5.1.4]{KPS18})
\[
L_{uv}=\{uv(su^4-tv^4)\}_{[s:t]\in \PP^1},\; u, v\in \PP(\bR_1), u\neq v
\]
or
\[
L_{u^2}=\{u^5(sx+ty)\}_{[s:t]\in \PP^1},\; u\in \PP(\bR_1).
\]
Hence, the fixed lines are $L_{xy}, L_{x^2}, L_{y^2}$, which match that of \eqref{fixedlines}.
\end{remark}

\subsection{Fixed conics in $X_5$}\label{fxcon}
The geometry of the Hilbert scheme of conics in the Grassmannian variety $\Gr(2, V_n)$, $\dim V_n=n$ has been intensively studied in \cite{HT15, CHL18, CM17}. The key observation is to consider the linear spanning of a conic in two different ways: The linear spanning of a conic in the Pl\"{u}cker embedding $\PP(\wedge^2 V_n)$ and the spanning of the moduli points of a conic in $\PP(V_n)$. For the del Pezzo threefold $X_5$, by considering the spanning as the moduli meaning, there exists an isomorphism
\[
\bH_2(X_5)\cong \Gr(4, V_5),
\]
where for a subvector space $V_4<V_5$, the conic is given by the intersection 
\begin{equation*}
 \Gr(2, V_4)\cap P\subset \Gr(2, V_5)\cap P=X_5.
\end{equation*}
For an explicit description of the above isomorphism, see \cite[Proposition 4.3]{CHL18}. Hence, the $\CC^*$-fixed conic is represented by the five vector spaces
\begin{gather*}
W_{ 4}=\{e_2, e_0, e_{-2}, e_{-4}\},\ 
W_{ 2}=\{e_4, e_0, e_{-2}, e_{-4}\},\\
W_{ 0}=\{e_4, e_2, e_{-2}, e_{-4}\},\ 
W_{-2}=\{e_4, e_2, e_0, e_{-4}\},\ 
W_{-4}=\{e_4, e_2, e_0, e_{-2}\}.
\end{gather*}
If $V_4=W_4$, the defining equation of the conic in $X_5$ is
\[
I_{\Gr(2, W_4)\cap P}=I_{X_5/\PP^9}+\langle p_{4, j}\ |\ j\in \{2, 0, -2, -4\}\rangle.
\]
In terms of $\SL_2$-orbit coordinates, the conic is defined by
\[
\langle a_{6}, a_{4}, a_{2}, a_{0}, a_{-2}^2\rangle.
\]
The other four cases are given by the following ones.
\begin{gather*}
\langle a_{6}, a_{2}, a_{0}, a_{-2}, a_{4}a_{-6}\rangle,\ 
\langle a_{4}, a_{2}, 8a_{0}^2+a_{6}a_{-6}, a_{-2}, a_{-4}\rangle,\\
\langle a_{6}a_{-4}, a_{2}, a_{0}, a_{-2}, a_{-6}\rangle,\ 
\langle a_{2}^2, a_{0}, a_{-2}, a_{-4}, a_{-6}\rangle.
\end{gather*}
%i9 : R=QQ[p_(4,2),p_(4,0),p_(4,-2),p_(4,-4),p_(2,0),p_(2,-2),p_(2,-4),p_(0,-2), p_(0,-4),p_(-2,-4),a_0..a_6]
%i10 : J=ideal(3*p_(2,0)-p_(4,-2), 2*p_(2,-2)-p_(4,-4), 3*p_(0,-2)-p_(2,-4),p_(4,2)*p_(0,-2)+p_(2,0)*p_(4,-2)-p_(4,0)*p_(2,-2),p_(4,2)*p_(0,-4)+p_(2,0)*p_(4,-4)-p_(4,0)*p_(2,-4),p_(4,2)*p_(-2,-4)+p_(2,-2)*p_(4,-4)-p_(4,-2)*p_(2,-4),p_(4,0)*p_(-2,-4)+p_(0,-2)*p_(4,-4)-p_(4,-2)*p_(0,-4),p_(2,0)*p_(-2,-4)+p_(0,-2)*p_(2,-4)-p_(2,-2)*p_(0,-4))
%i11 : K=ideal(p_(4,2),p_(4,0),p_(4,-2),p_(4,-4))
%i12 : C=J+K
%i15 : S=ideal(a_0-p_(4,2), 2*a_1-p_(4,0), a_2-p_(2,0), 3*a_2-p_(4,-2),2*a_3-p_(2,-2), 2*2*a_3-p_(4,-4), a_4-p_(0,-2), 3*a_4-p_(2,-4), 2*a_5-p_(0,-4), a_6-p_(-2,-4))
%i16 : U=eliminate(C+S,{p_(4,2),p_(4,0),p_(4,-2),p_(4,-4),p_(2,0),p_(2,-2),p_(2,-4),p_(0,-2), p_(0,-4),p_(-2,-4)})
%i26 : print toString U
%ideal(a_3,a_2,a_1,a_0,a_4^2)
%*********************************************************************************************************************
%i23 : I=ideal(a_0*a_4-4*a_1*a_3+3*a_2^2,a_0*a_5-3*a_1*a_4+2*a_2*a_3,a_0*a_6-9*a_2*a_4+8*a_3^2,a_1*a_6-3*a_2*a_5+2*a_3*a_4,a_2*a_6-4*a_3*a_5+3*a_4^2)
%i25 : isSubset(I,U)
%\[
%I_{X_5}+\langle p_{i,j}\ |\ j\neq i\rangle.
%\]

\subsection{Fixed twisted cubics in $X_5$}\label{sub:fixc}

In \cite[Proposition 2.46 and Remark 2.47]{San14}, the author proves that the Hilbert scheme of twisted cubic curves is isomorphic to \[\bH_3(X_5)\cong \Gr(2, V_5).\]Furthermore, he described the universal resolution of twisted cubics in $X_5$. By using a Schubert variety, the above isomorphism can be described by finding the defining equation of the twisted cubics (cf. \cite[Proposition 3.3]{Chu22}). In detail, let $\sigma_{1}(l)$ be the Schubert variety of lines meeting a fixed line $l$ in $\PP(V_5)$. The variety $\sigma_{1}(l)$ is a degree three, four-dimensional variety in $\PP(\wedge^2V_5)$. Hence, by cutting out by the codimension $3$ hyperplane $P$, we obtain a twisted cubic curve in $X_5$. 

Since the weights of $V_5$ are different (Lemma \ref{flocp}), there are ten $\CC^*$-fixed twisted cubics in $X_5$. For instance, let $l=\langle e_4, e_2\rangle$ be the line in $\PP(V_5)$. Then the lines meeting with the line $l$ would be generically described by the decomposable vectors
\begin{equation}\label{suber1}
(e_4+t_0 e_2)\wedge (e_2+t_1e_0+t_2e_{-2}+t_3 e_{-4})
\end{equation}
for $t_i\in \CC^*$ in $\PP(\wedge^2V_5)$. Therefore, the variety $\sigma_{1}(l)$ is the closure of the image of the map $(\CC^*)^4\lr \PP(\wedge^2V_5)$ given by equation \eqref{suber1}. That is,
\[
I_{\sigma_{1}(l)}=\langle \det \begin{bmatrix}p_{4,0}&p_{4,-2}&p_{4,-4}\\p_{2,0}&p_{2,-2}&p_{2,-4}\end{bmatrix}, p_{i,j} \rangle, \;i, j\in \{0, -2, -4\}.
\]
After cutting out by the codimension three linear subspace $P$ in \eqref{pluckeq}, we obtain the $\CC^*$-fixed twisted cubic $C_l$ in $X_5$, which presents the point $[l]$. In terms of orbit coordinates, the defining equation of $C_l$ is given by
\[
  I_{C_l/\PP^6}= \langle a_{-6}, a_{-4}, a_{-2}, a_0^2, a_2a_0, 3a_2^2-4a_4a_0 \rangle.\\
\]
%i48 : R=QQ[p_(4,2),p_(4,0),p_(4,-2),p_(4,-4),p_(2,0),p_(2,-2),p_(2,-4),p_(0,-2), p_(0,-4),p_(-2,-4),a_0..a_6]
%i49 : S=matrix{{p_(4,0),p_(4,-2),p_(4,-4)},{p_(2,0),p_(2,-2),p_(2,-4)}}
%i50 : Z=ideal(p_(0,-2), p_(0,-4),p_(-2,-4))
%i51 : T=minors(2, S)+Z
%i52 : P=ideal(3*p_(2,0)-p_(4,-2), 2*p_(2,-2)-p_(4,-4), 3*p_(0,-2)-p_(2,-4))
%i53 : C=T+P
%i54 : degree C, genus C
%i55 : S=ideal(a_0-p_(4,2), 2*a_1-p_(4,0), a_2-p_(2,0), 3*a_2-p_(4,-2),2*a_3-p_(2,-2), 2*2*a_3-p_(4,-4), a_4-p_(0,-2), 3*a_4-p_(2,-4), 2*a_5-p_(0,-4), a_6-p_(-2,-4))
%i56 : CO=eliminate(C+S,{p_(4,2),p_(4,0),p_(4,-2),p_(4,-4),p_(2,0),p_(2,-2),p_(2,-4),p_(0,-2), p_(0,-4),p_(-2,-4)})
%i64 :I=ideal(a_0*a_4-4*a_1*a_3+3*a_2^2,a_0*a_5-3*a_1*a_4+2*a_2*a_3,a_0*a_6-9*a_2*a_4+8*a_3^2,a_1*a_6-3*a_2*a_5+2*a_3*a_4,a_2*a_6-4*a_3*a_5+3*a_4^2)
%i65 : isSubset(I,CO)
%i66 : print toString CO

In the same manner, we have nine other $\CC^*$-fixed twisted cubics in $X_5$ in Table \ref{table_C3} and their configurations are presented in Figure \ref{Contw}. We remark that the multiple structure of a twisted cubic curve of a given reduced curve is unique in each case (\cite[Lemma 3.2 and Lemma 3.5]{Chu22}).

\begin{table}[h]
   \setlength{\tabcolsep}{4pt}
   \centering
   \caption{$\mathbb{C}^{\ast}$-fixed twisted cubics}
   \label{table_C3}
   \begin{tabular}{|c|c|c|c|} 
    \hline
    \rule{0pt}{2.6ex} % after hline
    \emph{No.} & $l$ & $I_{C_l/\PP^6}$ &  \\ [0.5ex] % before hline
    \hline
    &  &  & \\ [-1.5ex]
    \emph{1} & $\langle e_2, e_{-2} \rangle$ & $\langle a_{-4},a_0,a_4,3a_{-2}^2+a_2a_{-6},9a_2a_{-2}-a_6a_{-6},3a^2_2+a_6a_{-2} \rangle$ & $C_3$ \\[1.5ex]
    \emph{2} & $\langle e_2, e_{-4} \rangle$ & $\langle a_{-2},a_2,a_4,a_0a_{-4},a_6a_{-4},8a^2_0+a_6a_{-6} \rangle$ & $l_2+C_2$ \\[1.5ex]
    \emph{3} & $\langle e_4, e_{-2} \rangle$ & $\langle a_{-4},a_{-2},a_2,a_4a_{-6},a_4a_0,8a^2_0+a_6a_{-6} \rangle$ & $l_1+C_2$ \\[1.5ex]
    \emph{4} & $\langle e_4, e_{-4} \rangle$ & $\langle a_{-2},a_0,a_2,a_4a_{-6},a_6a_{-6},a_6a_{-4} \rangle$ & $l_0+l_1+l_2$ \\[1.5ex]
    \emph{5} & $\langle e_0, e_{-4} \rangle$ & $\langle a_0,a_2,a_6,a_4a_{-6},a_4a_{-2},a^2_{-2} \rangle$ & $l_0+2l_2$ \\[1.5ex]
    \emph{6} & $\langle e_4, e_0 \rangle$ & $\langle a_{-6},a_{-2},a_0,a_2a_{-4},a_6a_{-4},a^2_2 \rangle$ & $l_0+2l_1$ \\[1.5ex]
    \emph{7} & $\langle e_0, e_{-2} \rangle$ & $\langle a_{-2},a_0,a_6,a_2a_{-6},a^2_2,3a_2a_{-4}-a_4a_{-6} \rangle$ & $2l_0+l_2$ \\[1.5ex]
    \emph{8} & $\langle e_2, e_0 \rangle$ & $\langle a_{-6},a_0,a_2,a^2_{-2},a_6a_{-2},3a_4a_{-2}-a_6a_{-4} \rangle$  & $2l_0+l_1$ \\[1.5ex]
    \emph{9} & $\langle e_{-2}, e_{-4} \rangle$ & $\langle a_2,a_4,a_6,a_0a_{-2},a^2_0,3a_{-2}^2-4a_0a_{-4} \rangle$ & $3l_2$ \\[1.5ex]
    \emph{10} & $\langle e_4, e_2 \rangle$ & $\langle a_{-6}, a_{-4}, a_{-2}, a_0^2, a_2a_0, 3a_2^2-4a_4a_0 \rangle$ & $3l_1$ \\[1.5ex]
    \hline
 \end{tabular}  
 \end{table} 
 
%JS: PrimaryDecomposition Check
%1
%primaryDecomposition ideal(a_6,a_5,a_4,(a_3)^2,a_2*a_3,3*(a_2)^2-4*a_1*a_3)
%2: double line and line
%primaryDecomposition ideal(a_6,a_4,a_3,a_2*a_5,a_0*a_5,(a_2)^2)
%3: conic and line
%primaryDecomposition ideal(a_5,a_4,a_2,a_1*a_6,8*(a_3)^2+a_0*a_6,a_1*a_3)
%4: three lines
%primaryDecomposition ideal(a_4,a_3,a_2,a_1*a_6,a_0*a_6,a_0*a_5)
%5
%primaryDecomposition ideal(a_6,a_3,a_2,(a_4)^2,3*a_1*a_4-a_0*a_5,a_0*a_4)
%6: complicated
%primaryDecomposition ideal(a_5,a_3,a_1,3*a_4^2+a_2*a_6,9*a_2*a_4-a_0*a_6,3*a_2^2+a_0*a_4)
%7: conic and line
%primaryDecomposition ideal(a_4,a_2,a_1,a_3*a_5,a_0*a_5,8*(a_3)^2+a_0*a_6)
%8
%primaryDecomposition ideal(a_4,a_3,a_0,a_2*a_6,3*a_2*a_5-a_1*a_6,(a_2)^2)
%9: double line and line
%primaryDecomposition ideal(a_3,a_2,a_0,a_1*a_6,(a_4)^2,a_1*a_4)
%10:
%primaryDecomposition ideal(a_2,a_1,a_0,3*(a_4)^2-4*a_3*a_5,a_3*a_4,(a_3)^2)

%In summery, we have a configuration of $\CC^*$-fixed twisted cubics in $X_5$ as follows.

\begin{figure}[ht]
\caption{Configuration of $\CC^*$-fixed twisted cubic curves}
\vspace{2ex}
\label{Contw}
\begin{tikzpicture}[scale=0.83]
%%%%% 1 C3
\begin{scope}[xshift=0cm,yshift=0cm]
\draw[color=lightgray] (0.0,0.0)--(0.5,3.0);
\draw[color=lightgray] (3.0,0.0)--(2.5,3.0);
\draw[color=lightgray] (-0.3,0.5)--(3.3,0.5);
\draw[color=lightgray] (1.5,2.5) ellipse (1.07 and 0.5);
%\draw[color=lightgray] (-0.2,3.0) .. controls (0.5,2.1) and (2.5,2.1) .. (3.2,3.0);

\draw[thick] (-0.2,3.0) .. controls (0.5,2.1) and (2.5,2.1) .. (3.2,3.0);

\node at (1.5,-0.2) {\footnotesize 1. $C_3$};
\end{scope}

%%%%% 2 C2+l1
\begin{scope}[xshift=0cm,yshift=-4cm]
\draw[color=lightgray] (0.0,0.0)--(0.5,3.0);
%\draw[color=lightgray] (3.0,0.0)--(2.5,3.0);
\draw[color=lightgray] (-0.3,0.5)--(3.3,0.5);
%\draw[color=lightgray] (1.5,2.5) ellipse (1.07 and 0.5);
\draw[color=lightgray] (-0.2,3.0) .. controls (0.5,2.1) and (2.5,2.1) .. (3.2,3.0);

\draw[thick] (3.0,0.0)--(2.5,3.0);
\draw[thick] (1.5,2.5) ellipse (1.07 and 0.5);

\node at (1.5,-0.2) {\footnotesize 2. $l_2+C_2$};
\end{scope}

%%%%% 3 C2+l2
\begin{scope}[xshift=0cm,yshift=-8cm]
%\draw[color=lightgray] (0.0,0.0)--(0.5,3.0);
\draw[color=lightgray] (3.0,0.0)--(2.5,3.0);
\draw[color=lightgray] (-0.3,0.5)--(3.3,0.5);
%\draw[color=lightgray] (1.5,2.5) ellipse (1.07 and 0.5);
\draw[color=lightgray] (-0.2,3.0) .. controls (0.5,2.1) and (2.5,2.1) .. (3.2,3.0);

\draw[thick] (0.0,0.0)--(0.5,3.0);
\draw[thick] (1.5,2.5) ellipse (1.07 and 0.5);

\node at (1.5,-0.2) {\footnotesize 3. $l_1+C_2$};
\end{scope}

%%%%% 4 l0+l1+l2
\begin{scope}[xshift=4cm,yshift=0cm]
%\draw[color=lightgray] (0.0,0.0)--(0.5,3.0);
%\draw[color=lightgray] (3.0,0.0)--(2.5,3.0);
%\draw[color=lightgray] (-0.3,0.5)--(3.3,0.5);
\draw[color=lightgray] (1.5,2.5) ellipse (1.07 and 0.5);
\draw[color=lightgray] (-0.2,3.0) .. controls (0.5,2.1) and (2.5,2.1) .. (3.2,3.0);

\draw[thick] (0.0,0.0)--(0.5,3.0);
\draw[thick] (3.0,0.0)--(2.5,3.0);
\draw[thick] (-0.3,0.5)--(3.3,0.5);

\node at (1.5,-0.2) {\footnotesize 4. $l_0+l_1+l_2$};
\end{scope}

%%%%% 5 l0+2l2
\begin{scope}[xshift=4cm,yshift=-4cm]
\draw[color=lightgray] (0.0,0.0)--(0.5,3.0);
%\draw[color=lightgray] (3.0,0.0)--(2.5,3.0);
%\draw[color=lightgray] (-0.3,0.5)--(3.3,0.5);
\draw[color=lightgray] (1.5,2.5) ellipse (1.07 and 0.5);
\draw[color=lightgray] (-0.2,3.0) .. controls (0.5,2.1) and (2.5,2.1) .. (3.2,3.0);

\draw[thick] (3.0,0.0)--(2.5,3.0);
\draw[thick] (3.0+0.06,0.0)--(2.5+0.06,3.0);
\draw[thick] (-0.3,0.5)--(3.3,0.5);

\node at (1.5,-0.2) {\footnotesize 5. $l_0+2l_2$};
\end{scope}

%%%%% 6 l0+2l1
\begin{scope}[xshift=4cm,yshift=-8cm]
%\draw[color=lightgray] (0.0,0.0)--(0.5,3.0);
\draw[color=lightgray] (3.0,0.0)--(2.5,3.0);
%\draw[color=lightgray] (-0.3,0.5)--(3.3,0.5);
\draw[color=lightgray] (1.5,2.5) ellipse (1.07 and 0.5);
\draw[color=lightgray] (-0.2,3.0) .. controls (0.5,2.1) and (2.5,2.1) .. (3.2,3.0);

\draw[thick] (0.0,0.0)--(0.5,3.0);
\draw[thick] (0.0-0.06,0.0)--(0.5-0.06,3.0);
\draw[thick] (-0.3,0.5)--(3.3,0.5);

\node at (1.5,-0.2) {\footnotesize 6. $l_0+2l_1$};
\end{scope}

%%%%% 7 2l0+l2
\begin{scope}[xshift=8cm,yshift=-4cm]
\draw[color=lightgray] (0.0,0.0)--(0.5,3.0);
%\draw[color=lightgray] (3.0,0.0)--(2.5,3.0);
%\draw[color=lightgray] (-0.3,0.5)--(3.3,0.5);
\draw[color=lightgray] (1.5,2.5) ellipse (1.07 and 0.5);
\draw[color=lightgray] (-0.2,3.0) .. controls (0.5,2.1) and (2.5,2.1) .. (3.2,3.0);

\draw[thick] (3.0,0.0)--(2.5,3.0);
\draw[thick] (-0.3,0.5+0.03)--(3.3,0.5+0.03);
\draw[thick] (-0.3,0.5-0.03)--(3.3,0.5-0.03);

\node at (1.5,-0.2) {\footnotesize 7. $2l_0+l_2$};
\end{scope}

%%%%% 8 2l0+l1
\begin{scope}[xshift=8cm,yshift=-8cm]
%\draw[color=lightgray] (0.0,0.0)--(0.5,3.0);
\draw[color=lightgray] (3.0,0.0)--(2.5,3.0);
%\draw[color=lightgray] (-0.3,0.5)--(3.3,0.5);
\draw[color=lightgray] (1.5,2.5) ellipse (1.07 and 0.5);
\draw[color=lightgray] (-0.2,3.0) .. controls (0.5,2.1) and (2.5,2.1) .. (3.2,3.0);

\draw[thick] (0.0,0.0)--(0.5,3.0);
\draw[thick] (-0.3,0.5+0.03)--(4.3-1,0.5+0.03);
\draw[thick] (-0.3,0.5-0.03)--(4.3-1,0.5-0.03);

\node at (1.5,-0.2) {\footnotesize 8. $2l_0+l_1$};
\end{scope}

%%%%% 9 4l2
\begin{scope}[xshift=12cm,yshift=-4cm]
\draw[color=lightgray] (0.0,0.0)--(0.5,3.0);
%\draw[color=lightgray] (3.0,0.0)--(2.5,3.0);
\draw[color=lightgray] (-0.3,0.5)--(3.3,0.5);
\draw[color=lightgray] (1.5,2.5) ellipse (1.07 and 0.5);
\draw[color=lightgray] (-0.2,3.0) .. controls (0.5,2.1) and (2.5,2.1) .. (3.2,3.0);

\draw[thick] (3.0,0.0)--(2.5,3.0);
\draw[thick] (3.0-0.06,0.0)--(2.5-0.06,3.0);
\draw[thick] (3.0+0.06,0.0)--(2.5+0.06,3.0);

\node at (1.5,-0.2) {\footnotesize 9. $3l_2$};
\end{scope}

%%%%% 10 4l1
\begin{scope}[xshift=12cm,yshift=-8cm]
%\draw[color=lightgray] (0.0,0.0)--(0.5,3.0);
\draw[color=lightgray] (3.0,0.0)--(2.5,3.0);
\draw[color=lightgray] (-0.3,0.5)--(3.3,0.5);
\draw[color=lightgray] (1.5,2.5) ellipse (1.07 and 0.5);
\draw[color=lightgray] (-0.2,3.0) .. controls (0.5,2.1) and (2.5,2.1) .. (3.2,3.0);

\draw[thick] (0.0,0.0)--(0.5,3.0);
\draw[thick] (0.0-0.06,0.0)--(0.5-0.06,3.0);
\draw[thick] (0.0+0.06,0.0)--(0.5+0.06,3.0);

\node at (1.5,-0.2) {\footnotesize 10. $3l_1$};
\end{scope}

\end{tikzpicture}
\end{figure}
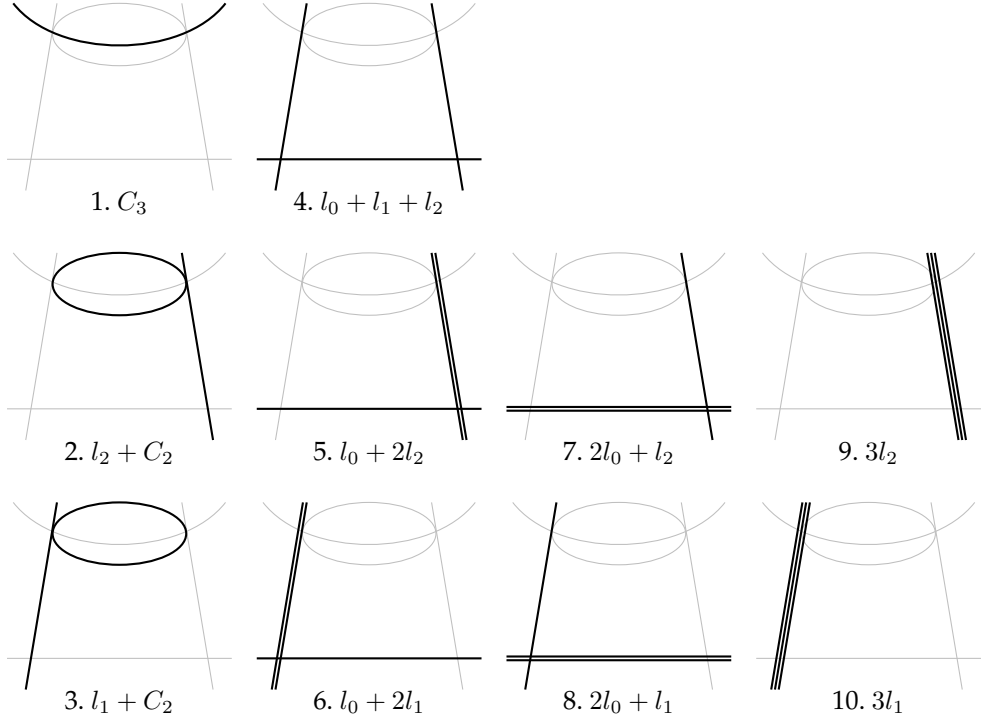

%\newpage
%\subsection{Geometry of $\bH_3(X_5)$}

\section{Rational and Elliptic quartic curves in $X_5$}\label{sec:raell}
As mentioned in the proof of \cite[Proposition 4.11]{FGP19}, for general $\PP^4\subset \PP^6=\PP(\Sym^6\CC^2)$, the (scheme-theoretic) intersection part $X_5\cap \PP^4=E_5$ is an elliptic quintic curve in $X_5$ by adjunction formula. However, if we choose every general linear subspace $\PP^4$ containing a line $L$ in $X_5$, then the intersection curve $E_5$ is a union of $L$ and a rational quartic curve $C_4$. Note that an irreducible rational quartic curve $C_4$ has a unique secant line in $X_5$ (\cite[Lemma 4.14]{FGP19}). In this section, we will describe the Hilbert scheme of rational quartic curves in $X_5$ by extending this geometric process.
%Related with the pair of linear subspaces in $\PP(V_6)=\PP(\Sym^6\CC^2)$, let us consider the partial flag variety $\text{Fl}(2, 5, V_6)$ which parametrizes the pair $(W_2, W_5)$, $W_2<W_5<V_6$ of subspaces in $V_6$. The Hilbert scheme $\bH_4(X)$ is birationally isomorphic to the fiber product
%\[
%\text{Fl}_{X}:=\bH_1(X_5)\times_{\Gr(2, V_7)}\text{Fl}(2, 5, V_6)
%\]
%where $\bH_1(X_5)\subset \Gr(2, V_6)$ is the natural inclusion and $\text{Fl}(2, 5, V_6)\lr \Gr(2, V_6)$ is the projection map of the first and last component. $\text{Fl}_{X}$ is a $\Gr(3, 5)$-fibration over $\bH_1(X_5)\cong \PP^2$. Geometrically, $\text{Fl}_{X}$ parametrizes the $\PP^4$ containing a line in $X_5$.
\subsection{Quartic curves via residual curve}\label{redsec}
Let us start with a simple observation of the hyperplane section of quintic del Pezzo varieties.
\begin{lemma}[\protect{\cite[Section 2]{Pro92}}]\label{hyerxq}
The (scheme-theoretical) hyperplane section $X_5\cap H$ of any hyperplane $H$ in $\PP^6$ is a non-degenerate, irreducible (possibly, singular) quintic del Pezzo surface.
\end{lemma}
\begin{proof}
Since $X_5$ is non-degenerate, the intersection $F=X_5 \cap H$ is a surface.
 Note that $[F]$ is primitive in $\Pic(X_5)$.
Because $X_5$ is smooth, an irreducible component $F_1$ of $F$ becomes an element of $\Pic(X_5)$, and hence, $[F_1]=[F]$.
That is, $X_5 \cap H$ is irreducible.
The non-degeneracy of $X_5 \cap H$ follows from Lemma 8.1 in \cite{BL13}.
\end{proof}

Let us describe the construction of a rational quartic curve using a sheaf-theoretical method. For a line $L\subset X_5=X$, let $s_1$, $s_2\in \rH^0(I_{L/X_5}(1))$ be the sub-linear section of $\rH^0(\cO_{X}(1))$ such that $s_1\wedge s_2 \neq 0$. Then it gives rise to an exact sequence
\begin{equation}\label{rese5}
0\lr \cO_{X_5}(-2)\stackrel{\psi=(s_1, s_2)}{\longrightarrow} \cO_{X_5}(-1)\oplus \cO_{X_5}(-1)\lr \coker(\psi)\lr 0.
\end{equation}
The scheme theoretic intersection $E_5=X_5\cap V(s_1)\cap V(s_2)$ in the irreducible surface $X_5\cap V(s_1)$ is a curve by Lemma \ref{hyerxq}. Therefore, by computing Hilbert polynomials of each term in \eqref{rese5}, we know that $\coker(\psi)\cong I_{E_5/X_5}$, $\chi(\cO_{E_5}(m))=5m$. The curve $E_5$ is connected by \cite{FH79} and clearly contains the line $L$ by its construction.
\begin{proposition}\label{idquot}
Under the above assumption, for any elliptic quintic curve $E_5$ in $X_5$ given a codimension two linear section of $X_5$ containing a line $L$, we have
\begin{equation}\label{gensred}
\Hom_{X_5} (\cO_L(-2), \cO_{E_5})\cong \CC\langle \phi\rangle.
\end{equation}
Furthermore, 
\begin{enumerate}
\item the non-zero map
\begin{equation}\label{phmap}
\phi: \cO_L(-2)\lr \cO_{E_5}
\end{equation}
is injective and thus the cokernel of $\phi$ is isomorphic to $\text{coker}(\phi)=\cO_{C_4}$ for a curve $C_4$ with Hilbert polynomial $\chi(\cO_{C_4}(m))=4m+1$ in $X_5$.
\item The image of the map $\phi$ in \eqref{phmap} is isomorphic to $\text{im}(\phi)\cong \cO_L(-1)$ if and only if $C_4$ is a non Cohen–Macaulay (CM)-curve.
\end{enumerate}
\end{proposition}
\begin{proof}
Through some diagram chasing by using the structure sequence: 
\begin{equation}\label{structure}
   \ses{I_{E_5/X_5}}{\cO_{X_5}}{\cO_{E_5}},
\end{equation}
we have $\Hom_{X_5} (\cO_L(-2), \cO_{E_5})\cong \Ext_{X_5}^1(\cO_L(-2), I_{E_5/X_5})$. Also, from the resolution \eqref{rese5} of $I_{E_5/X_5}$, there is an exact sequence
\begin{equation}\label{ie5}
   \begin{split}
0=\Ext_{X_5}^1(\cO_L(-2),\cO_{X_5}(-1)^{\oplus2})&\lr\Ext_{X_5}^1(\cO_L(-2), I_{E_5/X_5})\\ &\lr \Ext_{X_5}^2(\cO_L(-2), \cO_{X_5}(-2))=\CC\lr\cdots
\end{split}
\end{equation}
where the equality of the first term in \eqref{ie5} holds because of the Serre vanishing and the duality theorem, and the third one is 
\begin{equation*}
   \Ext_{X_5}^2(\cO_L(-2), \cO_{X_5}(-2))\cong \Ext_{X_5}^1(\cO_{X_5}, \cO_L(-2))^*\cong\rH^1(\cO_L(-2))^*\cong \rH^0(\cO_L).
\end{equation*}
For generic linear sections $\langle s_1, s_2\rangle$, the result \eqref{gensred} holds and thus it is true for any linear section.
 %사실은 위의 쉬컨스의 오르쪽 하나를 더 쓰고 나서 맵을 분석해보면 모두 직선을 포함하고 있기 때문에 0맵이 되어서 하나로 생성된다고 말할 수도 있다.

(1) The kernel of the map $\phi$ is of the form $\ker(\phi)=\cO_L(-k)$, $k \geq 3$, then $\text{im}(\phi)$ is a torsion subsheaf in $\cO_{E_5}$. However, this contradicts the fact that $E_5$ is a CM-curve because $E_5$ is a locally complete intersection. Hence, $\ker(\phi)=0$.

(2) If $\text{im}(\phi)=\cO_L(-1)$, then the quotient is $\cO_L(-1)/\cO_L(-2)\cong \CC_{p}\subset \cO_{C_4}$ for a point $p\in C_4$. 
Hence, $\cO_{C_4}$ has a torsion subsheaf.
Conversely, let $\bar{C}$ be the maximal CM-subcurve of the non-CM curve $C_4$. Since $X_5$ does not contain any plane, the curve $\bar{C}$ has Hilbert polynomial $\chi(\cO_{\bar{C}}(m))=4m+1$ or $4m$ by the degree-genus formula (i.e., any non-planar CM-curve holds the genus bound $g\leq \frac{(d-2)(d-3)}{2}$). For the former case, $\bar{C}=C_4$, which is a contradiction to the assumption. Hence, $\chi(\cO_{\bar{C}}(m))=4m$. In this case, by the purity of $E_5$, the kernel of the restriction map $\cO_{E_5}\twoheadrightarrow \cO_{\bar{C}}$ is a torsion-free sheaf supported on the line $L$. By comparing the Hilbert polynomials, we see that $\text{im}(\phi)=\cO_L(-1)$.
\end{proof}

From the proof of (2) of Proposition \ref{idquot}, we need to know the non-existence of CM-curves in $X_5$ with Hilbert polynomial $4m$, which we will discuss in the next subsection.
\subsection{Elliptic quartic curves in $X_5$}
In this subsection, we prove the non-existence of elliptic quartic curves in $X_5$.
\begin{proposition}\label{noncm4}
There does not exist \emph{any} CM-curve $E$ in $X_5$ with $\chi(\cO_E(m))=4m$.
\end{proposition}
\begin{proof}
Using \eqref{structure}, one can show that every elliptic quartic curve $E$ on $X_5$ is contained in a hyperplane section of $X_5$. We now show that such curves do not exist. Let us assume that the CM-curve $E$ is a complete intersection. Let $F=X_5\cap H$ be a hyperplane section of $X_5\subseteq \PP^6$. According to \cite[Section~2]{Pro92}, either $F$ is a normal quintic del Pezzo surface with at most Du Val singularities, or $F$ is a non-normal ruled surface. If $F$ is normal, then let $\psi:\widetilde{F}\to F$ be the minimal resolution of $F$ and let $H'=\psi^*\Oo_{F}(1)$. Note that $\widetilde{F}$ is a (weak) del Pezzo surface of degree~$5$. Suppose that $\widetilde{F}$ contains an elliptic curve $E$ whose image in $X_5$ has degree $4$, i.e., $H'\cdot E=4$. By the adjunction formula, we have $E^2=4$ since $-K_{\widetilde{F}}\cdot E=H'\cdot E=4$. This violates the Hodge index theorem ($16=(H'\cdot E)^2\geq (H'\cdot H')(E\cdot E)=5\times 4=20$). If $F$ is non-normal, then, according to \cite[Lemma~2.1]{Pro92}, there exists a morphism $\psi:\FF_m\to F$, where $\FF_m=\PP(\cO_{\PP^1}\oplus \cO_{\PP^1}(-m))$ is the $m$-th Hirzebruch surface for $m=1$ or $3$. We denote $H'=\psi^*\Oo_{F}(1)$, let $s$ be the negative section, and let $f$ be the fiber class of $\FF_m$. Note that $\psi$ is defined by a sublinear system of $|s+3f|$ if $m=1$, and of $|s+4f|$ if $m=3$. If $m=1$, then $H'=s+3f$ and $K_{\mathbb{F}_{1}}=-2s-3f$. For a curve $E$ with numerical class $E=as+bf$, the $H'$-degree $d(E)$ and the arithmetic genus $g(E)$ of $E$ are given by
\[
d(E)=2a+b \quad\text{and}\quad g(E)=\frac{1}{2} (a-1)(2b-a-2).
\]
Otherwise, if $m=3$, then $H'=s+4f$ and $K_{\mathbb{F}_{3}}=-2s-5f$. For a curve $E$ with numerical class $E= as+bf$, the $H'$-degree $d(E)$ and the arithmetic genus $g(E)$ of $E$ are given by
\[ 
d(E)=a+b \quad\text{and}\quad g(E)=\frac{1}{2} (a-1)(2b-3a-2).
\]
In either case, there are no integers $a$ and $b$ such that $d(E)=4$ and $g(E)=1$.
Lastly, if the CM-curve $E$ is not a complete intersection, then its defining equations are given by two quadrics having the common linear factor and a cubic form (see \cite[Proposition 2.5, (2) and (3)]{CCM16}). This implies that a plane is contained in $X_5$, which is absurd.

\end{proof}

\subsection{Construction of a morphism from a Grassmannian bundle to the Hilbert scheme}
\begin{proposition}\label{cmrat}
Let $C\subset X_5$ be a curve with Hilbert polynomial $\chi(\cO_C(m))=4m+1$. Then,
\begin{enumerate}
\item the curve $C$ is CM and
\item the linear span of the curve $C$ is $\langle C\rangle =\PP^4$. 
\end{enumerate}
\end{proposition}
\begin{proof}
(1) By the proof of (2) in Proposition \ref{idquot}, whenever the curve $C$ is not a CM-curve, then its maximal CM-subcurve must be an elliptic quartic curve. Such curves do not exist in $X_5$ by Proposition \ref{noncm4}.

(2) Since $C$ is a CM-curve with $\chi(\cO_C)=1$, the structure sheaf $\cO_C$ is stable by Lemma 3.17 in \cite{CC11}. Clearly, $\langle C\rangle \not\subset \PP^2$. Assume that $\langle C\rangle =\PP^3$. In \cite[Theorem 6.1]{CCM16}, it was proved that the sheaf $\cO_C$ fits into an exact sequence
\begin{equation}\label{eq:resqur}
   0\lr \cO_{\PP^3}(-3)^{\oplus 3}\lr \Omega_{\PP^3}(-1)\lr \cO_{S}\lr \cO_C\lr 0
\end{equation}
for some quadric surface $S\subset \PP^3$. From the resolution of $\cO_C$ in \eqref{eq:resqur}, one can see that the dimension $h^0(I_{C/\PP^3}(2))=1$ for the ideal sheaf $I_{C/\PP^3}$. Among the hypersurfaces defined by the five quadric generators of $X_5$ in \eqref{oribeq}, at least one intersection with $\langle C\rangle = \PP^3$ yields a quadric surface. Since $h^0(I_{C/\PP^3}(2)) = 1$, this surface is unique, and hence the del Pezzo variety $X_5$ contains the quadric surface $S$. With respect to the rank of the quadric form $q$ defining $S$, one can easily see that it contradicts the geometry of $X_5$. In fact, if $\mathrm{rk}(q)=4$, then $\bH_1(S)\cong \PP^1 \sqcup \PP^1$ is contained in $\bH_1(X_5)\cong \PP^2$, which contradicts B\'ezout's theorem. If $\mathrm{rk}(q)=3$, then there exists an infinite family of lines passing through the vertex, contradicting the fact that there are at most three lines passing through a point in $X_5$ (\cite[Section 2.1]{Ili94}). Finally, suppose that $\mathrm{rk}(q)\le 2$. Then $\PP^2\subset S\subset X_5$, which is impossible.
\end{proof}
We are ready to propose one of the main ingredients in this paper. Let $\cL$ be the universal lines in $\bH_1(X_5)\times X_5$ with the projection map $\pi: \bH_1(X_5)\times X_5\lr \bH_1(X_5)$ into the first component. Let us consider the direct image sheaf $\cF=\pi_*\cI_{\cL}(1)$. Since $\rH^0(I_{L/X_5}(1))\cong \CC^5$ for each line $L$ in $X_5$, the sheaf $\cF$ is locally free on the Hilbert scheme $\bH_1(X_5)\cong \PP^2$ (\cite[Section 2]{FN89}). Let $p:\bG=\Gr(2, \cF)\lr \bH_1(X_5)$ be the relative Grassmannian bundle over $\bH_1(X_5)$ which parametrizes linear subspaces $\PP^4$ containing a line in $X_5$.  As a set, the Grassmannian bundle $\bG$ is an incident variety: 
\begin{equation*}
   \bG=\{(l, H)\mid l\subset H, H\cong \PP^4\subset \PP^6, [l]\in \bH_1(X_5)\}.
\end{equation*}
\begin{proposition}\label{bij_morph}
Under the above definition and notation, there exists an injective morphism from the flag variety $\bG$ to the Hilbert scheme $\bH_4(X_5)$ which associates $(l, H)$ to the support of the cokernel $\text{coker}\{\cO_L(-2)\stackrel{\phi}{\lr} \cO_{X_5\cap H}\}$ (Section \ref{redsec}).
\end{proposition}
\begin{proof}
Let $\cU$ be the universal subbundle on the relative Grassmannian $\bG$. We relativize the exact sequence in \eqref{rese5}. Let $\overline{\pi}:\bG\times X_5\lr \bG$ be the projection map into the first component. 
\[
\xymatrix{\bG\times X_5\ar[r]^{p\times \text{id}} \ar[d]_{\bar{\pi}}& \bH_1(X_5)\times X_5\ar[d]^{\pi}\ar[r]&X_5\\
\bG\ar[r]^{p}& \bH_1(X_5)&}
\]
From the natural map $\pi^*\pi_*\cI_{\cL}(1)\lr \cI_{\cL}(1)$ and by pulling back the canonincal subbundle $\cU\subset p^*\cF$ on the Grassmannian bundle $\bG$ along the map $\overline{\pi}$, we have a sheaf homomorphism
\[
\Phi:\overline{\pi}^*\cU \lr \overline{\pi}^*p^*\cF=(p\times \text{id})^*\pi^*\cF=(p\times \text{id})^*\pi^*\pi_*\cI_{\cL}(1)\lr (p\times \text{id})^*\cI_{\cL}(1)\subset \cO(1).
\]
The Koszul complex of $\overline{\pi}^*\cU(-1):= \overline{\pi}^{\ast}\mathcal{U} \otimes \mathcal{O}_{X_5}(-1)$ with the twisted map $\Phi$ provides a locally free resolution of elliptic quintic curves in $X_5$. That is, there exists a locally free resolution:
\[
0\lr \wedge^2 \overline{\pi}^*\cU (-2)\lr \overline{\pi}^*\cU(-1) \lr \cI_{\cE_5}\lr 0
\]
where $\cE_5$ is the universal elliptic quintic curve over $\bG\times X_5$. Also, we can glue maps in \eqref{phmap} as follows. Consider the relative Hom-sheaf
\[
\cM=\cH om_{\overline{\pi}}((p\times\text{id})^*\cO_{\cL}(-2),  \cO_{\cE_5})
\]
on $\bG$. As we have seen in the proof of Proposition \ref{idquot}, each fiber of $\cM$ is completely determined by the structure sheaf of a line. Therefore, $\cM\cong p^*\cO_{\PP^2}(k)$, $k\in \ZZ$ for the canonical projection map $p: \bG\lr \bH_1(X_5)\cong \PP^2$. From the adjunction formula, there exists an isomorphism:
\begin{equation}\label{idhom}
\begin{split}
\Hom_{\bG}(\cM, \cM)&=\Hom_{\bG}(\cO_{\bG}, \cM\otimes \cM^*)= \rH^0( \bG, \cH om_{\overline{\pi}}((p\times\text{id})^*\cO_{\cL}(-2),  \cO_{\cE_5}\boxtimes\cM^*))\\
&\cong \Hom_{\bG\times X_5}((p\times\text{id})^*\cO_{\cL}(-2),  \cO_{\cE_5}\boxtimes\cM^*).
\end{split}
\end{equation}
The identity map $\text{id}\in \Hom(\cM, \cM)$ in \eqref{idhom} induces a sheaf homomorphism
\[
\Psi: (p\times\text{id})^*\cO_{\cL}(-2)\boxtimes \cO_{\PP^2}(k) \lr \cO_{\cE_5}
\]
over $\bG\times X_5$. The cokernel of $\Psi$ parametrizes rational quartic curves in $X_5$. Hence, from the universal property of the Hilbert scheme, we have a morphism
\begin{equation}\label{biregu}
\bar{\Psi}:\bG\longrightarrow \bH_4(X_5).
\end{equation}
 For the injectivity of $\bar{\Psi}$, suppose that $C_4=C_4'$. Then since $\langle C_4\rangle=\langle C_4'\rangle$, they determine the same line.
\end{proof}

\section{Geometry of Hilbert scheme of rational quartic curves}\label{sec:georatio}
In this section, we will prove that $\bH_4(X_5)$ is smooth (Proposition \ref{pro:sm}) and irreducible (Proposition \ref{irred}). Hence, we conclude that the injective morphism $\bar{\Psi}$ in \eqref{biregu} is an isomorphism by Zariski's main theorem.

\subsection{$\CC^*$-fixed rational quartic curves}
\begin{lemma}
Let $\bG$ be the Grassmannian bundle over $\bH_1(X_5)\cong \PP^2$. Then the fixed loci of $\bG$ are  isolated. That is, it consists of $30$ points.
\end{lemma}
\begin{proof}
Let $C$ be $\CC^*$-fixed rational quartic curves in $X_5$. From Proposition \ref{cmrat}, $\langle C\rangle=\PP^4$. So the linear space $\PP^4$ is fixed. Furthermore, $X_5\cap \PP^4=C\cup L$ is also fixed. Therefore, we need to consider only fixed lines in $X_5$. The fixed lines are defined by the ideal in $\PP^6$ (Section \ref{sub:fixl}):
\[
l_0=\langle a_6, a_2, a_0, a_{-2}, a_{-6}\rangle, \ l_1= \langle a_2, a_0, a_{-2}, a_{-4}, a_{-6}\rangle, \ l_2=\langle a_{-2}, a_0, a_{2}, a_{4}, a_{6}\rangle.
\]
The fiber of $\bG$ over each line $l_i$ is the Grassmannian $\Gr(2, l_0)$ when we regard $l_i$ as the subspace of $\rH^0(\cO_{\PP^6}(1))$. If we count the multiplicity of weights in the space $\wedge^2 l_0$, one can easily see that every fixed locus is isolated by Lemma \ref{flocp}. The defining ideals are presented in Table \ref{table_C4}. Among the total of thirty cases, \emph{symmetric} ones are listed only once
 --- only the second case (denoted with $^\ast$) has no symmetry. Here, we mean symmetric if the defining ideals have the same generators after replacing $a_i$ with $a_{-i}$.
\end{proof}

\begin{table}[h]
   \setlength{\tabcolsep}{2.5pt}
   \centering
   \caption{$\mathbb{C}^{\ast}$-fixed reducible quartic curves}
   \label{table_C4}
   \begin{tabular}{|c|c|c|} 
    \hline
    \rule{0pt}{2ex} % after hline
    \emph{No.} & $I_{C/\PP^6}$ & $C$   \\ [0.2ex] % before hline
    \hline
    &  &   \\ [-2ex]
    \emph{1} & $\langle a_0,a_4,a_{-2}a_{-4},a_2a_{-4},a_6a_{-4},3a_{-2}^2+a_2a_{-6},9a_2a_{-2}-a_6a_{-6},3a_2^2+a_6a_{-2} \rangle$ &  $l_2+C_3$  \\[0.6ex] 
    \emph{\hspace{0.1cm}2$^\ast$} & $\langle a_{-2},a_2,a_4a_{-6},a_0a_{-4},a_4a_{-4},a_6a_{-4},8a_0^2+a_6a_{-6},a_4a_0 \rangle$  & $l_1+l_2+C_2$  \\[0.6ex]
    \emph{3} & $\langle a_{-2},a_2,a_4a_{-6},a_0a_{-4},a_6a_{-4},8a_0^2+a_6a_{-6},a_4a_0,a_6a_4 \rangle$ & $l_0+l_2+C_2$ \\[0.6ex]
    \emph{4} & $\langle a_2,a_4,a_0a_{-4},a_6a_{-4},a_{-2}^2,a_0a_{-2},a_6a_{-2},8a_0^2+a_6a_{-6} \rangle$  & $2l_2+C_2$  \\[0.6ex]
    \emph{5} & $\langle a_0,a_2,a_4a_{-6},a_6a_{-6},a_6a_{-4},a_{-2}^2,a_4a_{-2},a_6a_{-2} \rangle$  & $l_0+l_1+2l_2$  \\[0.6ex]
    \emph{6} & $\langle a_0,a_2,a_{-2}a_{-6},a_4a_{-6},a_6a_{-6},a_{-2}^2,3a_4a_{-2}-a_6a_{-4},a_6a_{-2} \rangle$ & $2l_0+l_1+l_2$  \\[0.6ex]
%    \emph{7$^{\ }$} & $\langle a_{-2},a_0,a_2a_{-6},a_6a_{-6},3a_2a_{-4}-a_4a_{-6},a_6a_{-4},a_2^2,a_6a_2 \rangle$ &  $2l_0+l_1+l_2$  \\[0.6ex]
    \emph{7} & $\langle a_{-2},a_4,a_{-4}^2,4a_0a_{-4}-a_2a_{-6},a_2a_{-4},8a_0^2+a_6a_{-6},2a_2a_0+a_6a_{-4},a_2^2 \rangle$  & $2C_2$  \\[0.6ex]
%    \emph{9$^{\ }$} & $\langle a_{-4},a_2,a_{-2}^2,2a_0a_{-2}+a_4a_{-6},a_4a_{-2},8a_0^2+a_6a_{-6},4a_4a_0-a_6a_{-2},a_4^2 \rangle$ &  $2C_2$   \\[0.6ex]
    \emph{8} & $\langle a_{-2},a_6,4a_0a_{-4}-a_2a_{-6},3a_2a_{-4}-a_4a_{-6},a_0^2,a_2a_0,a_4a_0,a_2^2 \rangle$ &  $2l_0+2l_2$   \\[0.6ex]
    \emph{9} & $\langle a_0,a_2,a_4a_{-6},a_6a_{-6},a_{-2}^2,3a_4a_{-2}-a_6a_{-4},a_6a_{-2},a_6^2 \rangle$ & $2l_0+2l_2$ \\[0.6ex]
    \emph{10} & $\langle a_0,a_6,a_2a_{-6},3a_2a_{-4}-a_4a_{-6},a_{-2}^2,a_2a_{-2},a_4a_{-2},a_2^2 \rangle$  &  $2l_0+2l_2$  \\[0.6ex]
    \emph{11} & $\langle a_{-2},a_6,a_0a_{-6},4a_0a_{-4}-a_2a_{-6},3a_2a_{-4}-a_4a_{-6},a_0^2,a_2a_0,3a_2^2-4a_4a_0 \rangle$ & $3l_0+l_2$  \\[0.6ex]
    \emph{12} &$\langle a_2,a_6,a_4a_{-6},3a_{-2}^2-4a_0a_{-4},a_0a_{-2},a_4a_{-2},a_0^2,a_4a_0 \rangle$ &  $l_0+3l_2$ \\[0.6ex]
    \emph{13} & $\langle a_0,a_6,3a_2a_{-4}-a_4a_{-6},3a_{-2}^2+a_2a_{-6},a_2a_{-2},a_4a_{-2},a_2^2,a_4a_2 \rangle$  &  $l_0+3l_2$ \\[0.6ex]
    \emph{14} & $\langle a_4,a_6,3a_{-2}^2-4a_0a_{-4}+a_2a_{-6},2a_0a_{-2}-3a_2a_{-4},a_2a_{-2},a_0^2,a_2a_0,a_2^2\rangle$  &$4l_2$  \\[0.6ex]
    \emph{15} & $\langle a_2,a_6,3a_{-2}^2-4a_0a_{-4},2a_0a_{-2}+a_4a_{-6},a_4a_{-2},a_0^2,a_4a_0,a_4^2 \rangle$ & $4l_2$ \\[0.6ex]
    \hline
 \end{tabular}  
 \end{table}

\begin{example}\label{ex:irqu}
There exists a unique $\CC^*$-fixed rational normal curve $C_4$ whose defining equation is given by
\begin{equation*}
   \begin{split}
I_{C_4/\PP^6}=\langle a_{-6}, a_6, 3a_{-2}^2-4a_0a_{-4}, &2a_0a_{-2}-3 a_2a_{-4}, a_2a_{-2}-2a_4a_{-4}, \\ &4a_0^2-9a_4a_{-4}, 2a_2a_0-3a_{4}a_{-2}, 3a_2^2-4a_4a_0
\rangle,
\end{split}
\end{equation*}
where the degree two parts of the ideal $I_{C_4/\PP^6}$ are given by $2\times 2$-minors of the matrix
\[\begin{bmatrix}
\frac{27}{16}a_4 & \frac{9}{8}a_2 & a_0 & a_{-2}\\
\frac{9}{8}a_2 & a_0 & a_{-2} & \frac{4}{3}a_{-4}
\end{bmatrix}.
\]
\end{example}
In Section \ref{sec:rat}, we have already analyzed all of $\CC^*$-fixed rational curves of degree $\leq3$. Hence, by Example \ref{ex:irqu}, the reducible fixed curve up to the degree $d\leq 4$ has the configurations as in Figure \ref{fig:re4}.
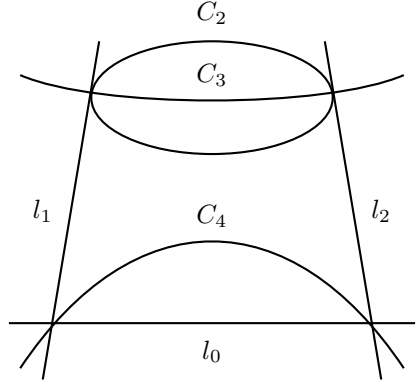
\begin{figure}[ht]
\caption{Reducible rational curves of degree $d\leq 4$}
\label{fig:re4}
\begin{tikzpicture}[scale=1.5]

\draw[thick] (0.0,0.0)--(0.5,3.0);
\draw[thick] (3.0,0.0)--(2.5,3.0);
\draw[thick] (-0.3,0.5)--(3.3,0.5);
\draw[thick] (1.5,2.5) ellipse (1.07 and 0.5);
\draw[thick] (-0.2,2.7) .. controls (0.5,2.4) and (2.5,2.4) .. (3.2,2.7);
\draw[thick] (-0.2,0.1) .. controls (0.8,1.6) and (2.2,1.6) .. (3.2,0.1);

\node at (0.0,1.5) {\footnotesize $l_1$};
\node at (3.0,1.5) {\footnotesize $l_2$};
\node at (1.5,0.25) {\footnotesize $l_0$};
\node at (1.5,3.25) {\footnotesize $C_2$};
\node at (1.5,2.7) {\footnotesize $C_3$};
\node at (1.5,1.45) {\footnotesize $C_4$};

\end{tikzpicture}
\end{figure}

%i1 : R=QQ[a_0..a_6]
%i2 :I=ideal(a_0*a_4-4*a_1*a_3+3*a_2^2,a_0*a_5-3*a_1*a_4+2*a_2*a_3,a_0*a_6-9*a_2*a_4+8*a_3^2,a_1*a_6-3*a_2*a_5+2*a_3*a_4,a_2*a_6-4*a_3*a_5+3*a_4^2)---X_5
%i3 : L=ideal(a_2,a_3,a_4,a_5,a_6)---line
%i4 : E=I+ideal(a_5,a_6)---elliptic
%i5 : S=E:L, hilbertPolynomial S
%i6:  hilbertPolynomial S----  -3P_0+4P_1으로 나와야함.

\begin{proposition}\label{pro:sm}
The Hilbert scheme $\bH_4(X_5)$ is a smooth variety.
\end{proposition}
\begin{proof}
For each $\CC^*$-fixed curve $C$ in Table \ref{table_C4}, one can check that the dimension of the deformation space at $C$ is $8$-dimensional by using \texttt{Macaulay2} (\cite{M2}). For the code used in our computation, see Example~\ref{excode} below. Also, in the next subsection, we will prove the irreducibility of $\bH_4(X_5)$. Hence, we finish the proof.
\end{proof}

\begin{example}\label{excode}
The dimension of the deformation space at the multiplicity $4$-line (No.~$15$ in Table~\ref{table_C4}) can be computed using the following \texttt{Macaulay2} code.
\lstset{
  basicstyle=\ttfamily\small,
  columns=fullflexible,
  keepspaces=true,
  breaklines=true,
  breakatwhitespace=true,
  showstringspaces=false
}
\begin{lstlisting}
R = QQ[a6,a4,a2,a0,am2,am4,am6];
--- Negative weights are denoted by 'm'
I = ideal(a6*am2 - 4*a4*a0 + 3*a2^2, a6*am4 - 3*a4*am2 + 2*a2*a0,
a6*am6 - 9*a2*am2 + 8*a0^2, a4*am6 - 3*a2*am4 +2*a0*am2, a2*am6 - 4*a0*am4 + 3*am2^2); 
--- The defining ideal of the quintic del Pezzo threefold
J = ideal(a2, a6, 3*am2^2 - 4*a0*am4, 2*a0*am2 + a4*am6, a4*am2, a0^2, a4*a0, a4^2);
--- defining ideal of the multiplicity 4 line
numcols basis(0,Hom(J/I,R/J))
--- The dimension of the deformation space at R/J
\end{lstlisting}
\end{example}

\begin{remark}
For a pair $(l ,\PP^4)$, $l\subset \PP^4$ we have unique rational quartic curve $C$. Let us assume that there exists a hyperplane $H(\cong \PP^5)$ which contains $\PP^4$ and the hyperplane section $H\cap X$ is a smooth del Pezzo surface. Then from the spectral sequence \cite[Lemma 4.16]{CKK24+}, one can easily check that the Hilbert scheme $\bH_4(X)$ is smooth at the closed point $[C]$.
\end{remark}

\subsection{Irreducibility of the Hilbert scheme}\label{sub:irred}

The locally free resolutions of rational curves of lower degrees in $X_5$ were presented in \cite{San14}. Let $\cU$ (resp. $\cQ$) be the universal sub(resp. quotient)-bundle on $\Gr(2, V_5)$.
We use the same notation to denote its restriction to $X_5$.
\begin{proposition}[\protect{\cite[Proposition 2.20, 2.32 and 2.46]{San14}}]\label{pro:rel3}
Under the above definition and notation, for each rational curve $C_d$ with degree $\deg (C_d)=d\leq 3$, the locally free resolution of an ideal sheaf $I_{C_d/X_5}$ is given by
\begin{enumerate}
\item ($d=1$) $0\lr\cU\lr\cQ^* \lr I_{C_1/X_5} \lr 0$,
\item ($d=2$) $0\lr\cO_{X_5}(-1)\lr\cU\lr I_{C_2/X_5} \lr 0$,
\item ($d=3$) $0\lr\cO_{X_5}(-1)^{\oplus 2}\lr\cQ(-1) \lr I_{C_3/X_5} \lr 0$.
\end{enumerate}
\end{proposition}
In particular, the resolution of a conic $C_2$ is a regular section of the dual bundle $\cU^*$ and thus its normal bundle in $X_5$ is isomorphic to $N_{C_2/X_5}\cong \cU^*|_{C_2}$ (\cite[Corollary 2.33]{San14}).
\begin{proposition}\label{irred}
The Hilbert scheme $\bH_4(X_5)$ of rational quartic curves in $X_5$ is irreducible.
\end{proposition}
\begin{proof}
Let $[C]\in \bH_4(X_5)$.
We denote by $C_{\text{red}}$ the reduced scheme of the curve $C$.
Note that the expected dimension of the Hilbert scheme $\bH_4(X_5)$ at $[C]$ is
\begin{equation*}
  \dim \Ext^0(I_{C/X_5}, \cO_C)-\dim \Ext^1(I_{C/X_5}, \cO_C)=8
\end{equation*}
by the stability of the structure sheaf $\cO_C$ (\cite[Lemma 3.17]{CC11}) and $K_{X_5}=-2H$.
Hence, each connected component of the Hilbert scheme $\bH_4(X_5)$ containing the curve $C$ has dimension at least $8$ (\cite[Theorem I.2.15]{Kol96}). Since the space of reduced rational quartic curves in $X_5$ is irreducible by \cite[Theorem 7.9]{LT19}, it is enough to prove that each locus of reducible curves with fixed types of a reduced curve has dimensions at most $7$. Note that the boundary of the Hilbert scheme $\bH_3(X_5)\cong \Gr(2, 5)$ is $5$-dimensional. We will estimate the upper bound of the dimension of multiple structures over a fixed reduced curve. For the detailed description of the multiple structure of a CM-curve in smooth projective threefolds, see the papers \cite{BF86, Nol97, NS03}.

\noindent
\textbf{Case 1.} $\deg(C_{\text{red}})=3$. Let us denote by $D=C_{\text{red}}$. If $p_a(D)\neq 0$, then $D$ is a planar curve with degree $3$, which violates the fact that $X_5$ does not contain any plane. Hence, $p_a(D)=0$ and thus structure sequence is given by
\begin{equation}\label{red3}
\ses{I_{D/C}}{\cO_C}{\cO_D}.
\end{equation}
Clearly, the Hilbert polynomial of the ideal sheaf $I_{D/C}$ is $\chi(I_{D/C}(m))=m$. Hence, it is supported on a line $L\subset C$. However, $C$ is a CM-curve and thus $I_{D/C}\cong \cO_L(-1)$. On the other hand, the non-split extensions of \eqref{red3} are  parametrized by \[\Ext^1(\cO_D, \cO_L(-1))\cong \Hom(I_{D/X_5}, \cO_L(-1))\] up to the scalar multiplication (cf. \cite[Example 2.1.12]{HuLe10}). However, one can easily see that 
\begin{equation*}
   \dim \Hom(I_{D/X_5}, \cO_L(-1))\leq 2
\end{equation*}
from the locally free resolution of $I_{D/X_5}$ in Proposition \ref{pro:rel3}.
Hence, the locus of non-split extensions of \eqref{red3} is at most $1$-dimensional.

\noindent
\textbf{Case 2.} $\deg(C_{\text{red}})=2$.
The possible support of $C_{\text{red}}$ is a smooth conic or a pair $L_1\cup L_2$ $L_1\neq L_2$ of lines meeting at a point.

\noindent
\textbf{Case 2-1.} $[C]=2 [Q]$ for some smooth conic $Q\subset X_5$. Such a locus is at most $1$-dimensional. The details are as follows. From the inclusion $Q\subset C$, there is a structure sequence
\begin{equation}\label{eq:qud}
\ses{\cO_Q(-p)}{\cO_C}{\cO_Q}
\end{equation}
for some $p\in Q$. Note that the line bundle $\cO_Q(-p)$
on $Q$ is isomorphic to each other for all $p\in Q$. Conversely, the long exact sequence (\cite[Lemma 4.16]{CKK24+}) 
\[
0=\Ext_Q^1(\cO_Q,\cO_Q(-p))\lr \Ext_{X_5}^1(\cO_Q,\cO_Q(-p)) \lr \rH^0(N_{Q/X_5}\otimes O_Q(-p))\cong\CC^2\lr
\cdots\]
says that non-split extensions in \eqref{eq:qud} are parametrized by $\PP^1$ or a point because $N_{Q/X_5}\cong \cU^*|_{Q}$ (Proposition \ref{pro:rel3}). Hence, we proved the claim.

\noindent
\textbf{Case 2-2.} $[C]=2[L_1]+2[L_2]$, $L_1\neq L_2$ for some pair of lines $L_1$ and $L_2$ meeting at a point. The double structure of $D=L_1\cup L_2$ is at most $3$-dimensional. In detail, since $D$ is a locally complete intersection, a double structure on the curve $D$ is determined by a surjective morphism
\begin{equation}\label{eq:sudou}
\psi: N_{D/X_5}^*=I_{D/X_5}/I_{D/X_5}^2\twoheadrightarrow \cL
\end{equation}
for a line bundle $\cL$ on $D$ with $\chi(\cL(m))=2m$ (\cite[Section 1]{BF86}). As tensoring the line bundle $\cL$ into the structure sequence $\ses{\cO_{L_2}(-1)}{\cO_D}{\cO_{L_1}}$, we have 
\begin{equation*}
   \ses{\cO_{L_2}(-a-2)}{\cL}{\cO_{L_1}(a)}, \ a\in \ZZ.
\end{equation*}
From the fact that $N_{D/X_5}\cong \cU^*|_D$ (Proposition \ref{pro:rel3}), $\Hom(N_{D/X_5}^*, \cO_{L_i}(k))\neq 0$ if and only if $k\geq -1$.
From this, one can check whenever $\psi$ is surjective, $a=-1$ or $0$.
If $a=-1$, then $\Hom(N_{D/X_5}^*, \cO_{L_i}(-1))\cong \CC$ because of $N_{D/X_5}\cong \cU^*|_D$ again. Hence in this case, $\dim \Hom(N_{D/X_5}^*, \cL)\leq 2$.
Note that if we replace the role of the lines $L_1$ and $L_2$, we may obtain another family of line bundles $\cL$. On the other hand, by upper semicontinuity and $\dim T_{[L_1]}\bH_1(X_5)=2$ (\cite[Section 2]{FN89}), we have 
\begin{equation*}
 \dim\Ext^1(\cO_{L_1}, \cO_{L_2})\leq \dim\Ext^1(\cO_{L_1}, \cO_{L_1})=2.
\end{equation*}
Hence, such a line bundle $\cL$ is at most $1$-dimensional.
After all, considering the case $a=0$ as well, the space of surjective morphisms $\phi$ in \eqref{eq:sudou} is at most $3$-dimensional.%Similarly, if $a=0$, then $\Hom(N_{D/X}^*, \cO_{L_i})=\Hom(\cO_{L_i}(-1)\oplus\cO_{L_i}, \cO_{L_i}(-1))\cong \CC^3$ because of $N_{D/X}\cong \cU^*|_D$ again. Hence in this case, $\dim \Hom(N_{D/X}^*, \cL)\leq 2$.
%$\dim\Ext^1(\cO_{L_1}, \cO_{L_2}(-2))\leq \dim\Ext^1(\cO_{L_1}, \cO_{L_1}(-2))=1$
\\
\noindent
\textbf{Case 3.} Let $[C]=3[L_1]+[L_2]$, $L_1\neq L_2$ for a pair of lines $L_1$ and $L_2$ meeting at a point. Such a locus is at most $4$-dimensional. Let us consider degree $3$ maximal CM-subcurve $C_3\subset C$ supported on $L_1$ which is obtained from the ideal quotient $(I_{C/X_5}: I_{L_2/X_5})$. Since the length of the intersection $C_3\cap L_2$ is $1\leq l(C_3\cap L_2)\leq 3$, the Hilbert polynomial of the curve $C_3$ is $\chi(\cO_{C_3}(m))=3m+k$, $1\leq k\leq 3$ (cf. \cite[Section 5]{NS03}). Also, the possible CM-filtrations of the curve $C_3$ are of types (cf. \cite[Section 2]{Nol97}): 
\begin{enumerate}[(a)]
\item $(m+1, 2m+1, 3m+1)$,
\item $(m+1, 2m+1, 3m+2)$,
\item $(m+1, 2m+1, 3m+3)$,
\item $(m+1, 2m+2, 3m+3)$,
\end{enumerate}
where each component of the $3$-tuple indicates the Hilbert polynomial of sub-curves in the filtration of $C_3$.
Note that a planar double line (i.e., Hilbert polynomial $2m+1$) exists only if the line is non-free (\cite[Section 1]{FN89}) and the non-planar double lines (i.e., Hilbert polynomial $2m+2$) are parametrized by the projectivization
 of the tangent space of a line in $X_5$.

For the  case (a), there is an exact sequence 
\begin{equation*}
 \ses{\cO_{L_2}(-1)}{\cO_{C}}{\cO_{C_3}},
\end{equation*} 
where $\chi(\cO_{C_3}(m))=3m+1$ and thus $C$ is determined by the same method as in Case 1. Note that the multiple curve $C_3$ supported on the non-free line $L_1$ is unique by \cite[Lemma 4.2]{Chu22}. Hence, the dimension of such a locus is $\leq 1$.

For the case (b) by diagram chasing, there is an exact sequence $\ses{F}{\cO_C}{\cO_{C_2}}$ where $F$ fits into the extension $\ses{\cO_{L_2}(-2)}{F}{\cO_{L_1}}$ and $C_2$ is a curve supported on $L_1$ with Hilbert polynomial $\chi(\cO_{C_2}(m))=2m+1$. In terms of ideal sheaves, we have an exact sequence $\ses{I_{C/X_5}}{I_{C_2/X_5}}{F}$. However, the map $I_{C_2/X_5}\lr F$ is at most $3$-dimensional because of $\Hom(I_{C_2/X}, \cO_{L_2}(-2))=0$ (Proposition \ref{pro:rel3}). Hence, such a non-trivial extension is at most $3+3=6$-dimensional. The same reasoning holds for the case (c).

For the case (d) by diagram chasing again, there exists an exact sequence 
\begin{equation*}
\ses{G}{\cO_C}{\cO_{C_2}}, 
\end{equation*}
where $G$ fits into the exact sequence 
$\ses{\cO_{L_2}(-3)}{G}{\cO_{L_1}}$ with $\chi(\cO_{C_2}(m))=2m+2$. In terms of ideal sheaves, we have an exact sequence
\begin{equation}\label{eq:3m+3}
\ses{I_{C/X_5}}{I_{C_2/X_5}}{G}.
\end{equation}
Conversely, we estimate the dimension of $\Hom(I_{C_2}, G)$. From the structure ideal sequence $\ses{I_{C_2/X_5}}{I_{L_1/X_5}}{\cO_{L_1}}$, one can easily see that
\[\dim \Hom(I_{C_2/X_5}, \cO_{L_1})\leq 2 \ \text{and} \ \dim \Hom(I_{C_2/X_5}, \cO_{L_2}(-3))\leq 1.\]
Therefore, the dimension of $C$'s in \eqref{eq:3m+3} is at most 
\begin{equation*}
 \dim \Hom(I_{C_2/X_5}, G)+\dim \PP\Ext^1(\cO_{L_1}, \cO_{L_1})\leq(2+1)+1=4.
\end{equation*}

%R=QQ[a_0..a_6]
%I=X_5 ideal
%M=R/I
%X=Proj (M)
%J=triple line ideal
%L=line ideal
%SJ=sheaf module J
%OL=sheaf module M/L
%Hom(SJ, OL(-1))

\noindent
\textbf{Case 4.} $\deg(C_{\text{red}})=1$.
Clearly, $C_{\text{red}}$ is a line. There exists a one-dimensional family of quadruple structure on a non-free line in $X_5$. This case is completely analogous to the first paragraph of the proof in \cite[Proposition 4.14]{CKK24+}. So we skip the proof.
\end{proof}

%\begin{corollary}
%For each curve $[C]\in \bH_4(X_5)$, $\rH^1(I_{C/X}(2))=0$. That is, the curve $C$ is $2$-regular in its spanning space $\langle C\rangle =\PP^4$ and thus $C$ is aCM-curve.
%\end{corollary}
%\begin{proof}
%앞의 모든 분류법들은 토러스 픽스드가 아니더라도 성립하는 것이므로 위 명제는 참이다.
%\end{proof}

%i1 : R=QQ[a_0..a_6]
%i2 :I=ideal(a_0*a_4-4*a_1*a_3+3*a_2^2,a_0*a_5-3*a_1*a_4+2*a_2*a_3,a_0*a_6-9*a_2*a_4+8*a_3^2,a_1*a_6-3*a_2*a_5+2*a_3*a_4,a_2*a_6-4*a_3*a_5+3*a_4^2)---X_5
%i3 : L=ideal(a_2,a_3,a_4,a_5,a_6)---line
%i4 : E=I+ideal(a_5,a_6)---elliptic
%i5 : S=E:L, hilbertPolynomial S
%i6:  hilbertPolynomial S----3P_0+4P_1으로 나와야함.

%\begin{proposition}
%Let $[C]\in \bH_4(X_5)$ assuming that $C$ is a CM-curve\footnote{사실은 모든 곡선은 CM이다고 주장할 것이다.}. Then $\langle C\rangle =\PP^4$.
%\end{proposition}
%\begin{proof}
%It is well-known that $\langle C\rangle=\PP^3$ or $\PP^4$. If the first case occurs, then from \cite[Theorem 6.1]{CCM16} and \cite[Theorem 4.9]{CN12}, one can read that the defining equation of the curve $C$ is generated by one quadric and three cubic forms because the structure sheaf $\cO_C$ is stable. This implies that the quadric surface is contained in $X_5$ which violates the geometry of lines and conics in $X_5$.
%\end{proof}

\end{document}